\documentclass[11pt,a4paper,twoside]{amsart}
\usepackage{amsmath}
\usepackage{amssymb}
\usepackage{amsfonts}
\usepackage{a4}

\newtheorem{theorem}{Theorem}[section]
\newtheorem{lemma}[theorem]{Lemma}

\theoremstyle{definition}
\newtheorem{definition}[theorem]{Definition}
\newtheorem{claim}[theorem]{Claim}

\newtheorem{question}[theorem]{Question}
\newtheorem{conjecture}[theorem]{Conjecture}
\newtheorem{remark}[theorem]{Remark}
 
\newtheorem{assertion}[theorem]{Assertion}

\newcommand{\mC}{{\mathbb C}}

\newcommand{\mG}{\mathbb G}

\newcommand{\mP}{\mathbb P}

\newcommand{\mA}{\mathbb A}

\newcommand{\bO}{\Omega}

\newcommand{\ep}{\epsilon}
\newcommand{\D}{\Delta}

\newcommand{\kk}{\kappa}

\newcommand{\mcF}{\mathcal F}

\newcommand{\mcN}{\mathcal N}

\newcommand{\mcS}{\mathcal S}

\newcommand{\fg}{\mathfrak g}

\newcommand{\ti}{\tilde}

\newcommand{\un}{\underline}

 \newcommand\supp{\mathrm{supp}}

\newcommand{\sm}{\setminus}

\newcommand{\sms}{\smallskip}

\usepackage{xcolor}
\usepackage{amscd} \title[A question about the Fourier transform]{ A question about the Fourier transform}
\author{ David Kazhdan}

\begin{document}
\maketitle

\begin{abstract}
This  paper formulates a conjectural description of 
of the space tres cuspidale functions 
 studied in \cite{W} (called 
weightless functions in \cite{BK}) and raises a question about a possibility of extending such a 
description in a more general context.
\end{abstract}

\section{}

Let $F$ be a local or a finite field. For an $F$ variety $\un X$ we write $X:= \un X (F)$ 
denote by 
$\mcS (X)$ 
the Schwartz space of complex valued functions on $X$. \begin{remark}If $F$ is non-archimedian,
  then $\mcS (X)$ is the space of locally constant
   compactly supported functions on $X(F)$ 
  and if $F$ is archimedian the space 
$\mcS (X)$ is defined in \cite{AG}.

\sms

We fix a non-trivial additive character
$\psi :F\to \mC ^\ast$ and 
 for a $F$-vector space  denote by $\mcF :\mcS (V  ^\vee)\to \mcS (V)$ the Fourier transform, where $ V ^\vee $ is the dual vector space.

\end{remark} \begin{definition}
\begin{enumerate}
\item  $\un G$ is  a split semisimple group and  $G= \un G (F)$.
\item  $\fg$  is the Lie algebra of
$\un G$.

\item 
$\mathcal U$ is the variety  of
   unipotent radicals of proper parabolic subgroups of $\un G$ \footnote { $\mathcal U$ is simultoneously the set of
   unipotent radicals of proper parabolic subalgebras of $\fg$.}.

\item For  a conjugacy class $\un  \bO \subset \un G$ 
we denote by $\un  Y^\bO \subset \un G\times \mathcal U $ 
the subvariety  of pairs $ (g,U) $ such that
 $ g U\subset \bar {\un  \bO }$ where $ \un {\bar \bO} $ is 
the closure of $\un  \bO$.

\item For  a conjugacy class $\bO \subset \fg$ we denote by $ Y^\bO \subset \fg \times \mathcal U $ the subset of pairs $ (x,N) $ such that $ x+N\subset \bar \bO $.  \end{enumerate}
,
\end{definition}

\section{} In this section $F$  is a finite field.

\begin{definition}
\begin{enumerate}
\item $\mcS ( G)_{cusp}\subset \mcS (G) $ is the subspace of cuspidal  functions, that is of functions $f$ on $G$ such that $\sum_{u \in U}f(gu)=0$ for all $g\in G,u\in \mathcal U $.

\item    $\mcS _{w}(\bar  \bO)\subset \mcS (\bar  \bO) $ is the subspace of functions
  $f$ such that $\sum_{u\in U}f(gu)=0$ for all $g\in G,u\in \mathcal U $. 

\item  $\mcS (\bar  \bO) _{cusp} \subset \mcS (\bar  \bO) $
  is the space of restrictions of cuspidal functions onto $\bar  \bO$.
  \end{enumerate}

\end{definition}

It is clear that $\mcS _{cusp} ( \bar \bO)\subset \mcS _{w} ( \bar \bO) $.

\begin{conjecture}\label{1} $\mcS _{cusp} (\bar  \bO) = \mcS _{w} (\bar  \bO) $ for any conjugacy class $\un  \bO \subset  \mG$.

\end{conjecture} 
\section{} In this section $F$  is a  local non-archimedian field.

\begin{definition}
\begin{enumerate}

\item $\mcS (G)_{cusp}\subset \mcS (G) $ is the subspace of cuspidal functions, that is of functions $f\in \mcS (G) $ such that
  $\int _{u \in U}f(gu)du =0$ for $(g,U)\in Y$.

\item   Let $\bO \subset G$ be a conjugacy class and $\bar \bO$ be the closure of $\bO$. Then
$\mcS _{w}(\bar \bO)\subset \mcS (\bar  \bO) $ is the subspace of functions
  $ f $ such that $\int _{u\in U}f(gu)du=0$ for all $(g,U) \in Y^\bO$. 

\item  $\mcS ( \bar \bO) _{cusp}$ is the space of restrictions of cuspidal functions
  onto $\bar \bO$. \end{enumerate}

\end{definition}

It is clear that $\mcS _{cusp} (\bar \bO)\subset \mcS _{w} (\bar \bO) $.

\begin{conjecture}\label{1} $\mcS _{cusp} (\bar \bO) = \mcS _{w} (\bar \bO) $ for any conjugacy class $\bO \subset G$.

\end{conjecture}  

\section{} This section is on a formulation of a Lie algebra analogue of Conjecture \ref{1} which also make sense 
for archimedian fields.

\begin{definition}
\begin{enumerate}

 \item $\fg  ^\vee _{ell}\subset \fg  ^\vee $ is the open subset of elements whose stabilizer in $G$
is an anisotropic  torus.

\item A function $f\in \mcS(\fg) $ is {\it cuspidal} if
$\supp (\mcF (f))\subset \fg  ^\vee_{ell} $ where $\mcF :\mcS (\fg)\to \mcS(\fg  ^\vee)$ is the Fourier transform.

\item For $f\in \mcS(\bar \bO),y=(x,N)\in Y^\bO $ we denote by 
$f_y: N\to \mC$ the function given by $f_y(n):= f(x+n)$.
\item  $\hat f_y \in \mcS (N^\vee)$ is the Fourier transform of the function $f_y$.
\item   $\mcS _{w}(\bar \bO)\subset \mcS (\bar \bO) $ is the subspace of functions $ f $ such that $ \hat f_y \in \mcS (N_y^\vee \sm 0)$ for all $ y=(x,N)\in Y^\bO $.

\item  $\mcS (\bar \bO) _{cusp} \subset \mcS (\bar \bO) $ is the space of restrictions onto $\bar \bO$ of cuspidal functions.

 \end{enumerate}

\end{definition}
\begin{remark} If  $F$ is either a non-archimedian or a finite field then   $\hat f_y  \in \mcS (N_y^\vee \sm 0)$ iff $ \int _{n\in N_y}f_y(n)dn=0 $.
\end{remark}

It is clear that $\mcS _{cusp} (\bar \bO)\subset \mcS _{w} (\bar \bO) $.

\begin{conjecture}\label{Lie} $\mcS _{cusp} (\bar  \bO) =
 \mcS _w( \bar \bO)$.

\end{conjecture} 

\section{} \subsection{}

This section contains a proof of Conjecture \ref{Lie} in the case when  $G=PGL(2,F)$ and 
 $F$ is either a finite or a local  non-archimedian field. I expect the possiblity  to extend the proof  to 
the case of Archimedian fields.

\sms

The proof is based on the following observation.

\begin{assertion} Let  $G=PGL(2,F)$ and $\bO \subset \fg$ be a conjugacy class. Then
$dim _GHom ( \mcS (\bO), \pi )\leq 1$ for any irreducible representation of $G$ different from the Steinberg representation $St$. \footnote { $St$  is the space of smooth measures on $\mP ^1$ of total volume zero.}
\end{assertion} \subsection{}

In the case when $\un G= PGL(2)$ the pairing $(x,y)\to tr (xy)$ identifies the dual space $\fg ^\vee$ with $sl_2$.

To simplify notations I present a proof only for 
the conjugacy class  $ \mcN $ of regular nilpotent elements, but 
analogous  arguments work  also in the case of  an arbitrary conjucacy class 
$\bO \subset \fg$.

 \subsection{}

I start with a reminder of  definitions and  results on representations of the group $ PGL(2,F) $.

\sms

\begin{definition}Let  $\chi : F^\ast \to \mC ^\ast$ be a character of $F^\ast$.
\begin{enumerate}
\item
$(\pi _\chi ,V_\chi)$ is the space of smooth functions $f$ on $\mcN$ such that $f(ax)= \chi (a)\|a\|^{-1/2}f(x), x\in \mcN ,a\in F^\ast$. 
\item
$\kk _\chi : \mcS (\mcN)\to  V _\chi (f)$ is the surjective morphism defined by $\kk _\chi (f)(x):=
\int _{a\in F^\ast}f(ax)\chi ^{-1}(a)\|a\|^{1/2}d^\ast a$.
\item $\mcS _- {:= }ker (\kk _{  \| \|^{- 1/2}})$.

\end{enumerate}
\end{definition}

\begin{claim}\label{as} \begin{enumerate}

\item For  $\chi \neq  \| \|^{\pm 1/2} $ the 
representation $V_\chi$ is irreducible.
\item $V_ {\| \|^{ 1/2} }$ is the space of smooth functions on $\mP ^1$.
\item  
The  representation $V_ {\| \|^{ 1/2} }$ is generated by any $v\in V_ {\| \|^{ 1/2} }\sm V_0$ where  $V_0\subset V$ be the subspace of constant functions.

\item If $R\subset \mcS _-$ is a $G$-invariant subspace such that
  $\kk _\chi (R)\neq \{0\}$ for
 $\kk \neq  \| \|^{ \pm 1/2} $ and 
  $\kk _{  \| \|^{ 1/2}} (R)\not \subset V_0$ then $R= \mcS _-$. 

 \item $ \mcS _- = \mcS _w$. \end{enumerate}
\end{claim}

\begin{lemma}\label{chi} 

 \begin{enumerate}

\item For any
 $\chi \neq  \| \|^{ \pm 1/2} $ there exists $h _\chi \in \mcS (\fg _{ell}^\vee)$ such that $ \mcF (h _\chi ) (0)=0$ and 
$\kk _\chi (\mcF (h _\chi  ))\neq 0.$  
\item If $F$ is a non-archimedian field then 
there exists $h\in \mcS (\fg _{ell}^\vee)$ such that $ \mcF (h) (0)=0$ and $\kk _{  \| \|^{ 1/2}}(\mcF (h))\not \in V_0$.

\end{enumerate}
\end{lemma} 

As follows from Claim \ref{as} Lemma \ref{chi} implies the validity of Conjecture \ref{Lie}.

\section{} This section presents a proof of Lemma \ref{chi}  for finite fields $F$. 

\sms

It is clearly  sufficient to show that for any 
non-trival character $\chi$ of $F^\ast$ there exists 
a function $h_\chi \in \mcS(\fg ^\vee _{ell})$ such $\kk _\chi (\mcF (h\chi )_{|\mcN})\neq 0$. We construct such a function explicitely.

\sms

Let  $x= \begin{pmatrix}0&1 \\\ep &0 \end{pmatrix} \in \fg ^\vee _{ell}, \ep \in F\sm F^2, z_u:= \begin{pmatrix}0&0 \\u&0 \end{pmatrix}\in \mcN ,u \in F^\ast $ and 
 $f_a,a\in F^\ast$ be the 
Fourier transform of the delta function $ \delta _{ax} $ . Since  $f_a (z_u) =\psi (au)$, we have $$\kk _\chi (f_a-f_b)
(z_1)= \sum _{ u  \in F^\ast} \psi (au)\chi (u^{-1})-  \sum _{ u  \in F^\ast} \psi (bu) \chi (u^{-1} )= (\chi (a^{-1})- \chi (b^{-1}))\D
$$ where   $\D :=  \sum _{ u  \in F^\ast} \psi (u) \chi (u^{-1}) $.

Choose  $a,b\in F^\ast$ such that $\chi (a/b)\neq 1$ and write $h _\chi := \delta _{ax}-\delta _{bx}$. Since $\D \neq 0 $ the function $h _\chi $ satisfes  the conditions of Lemma    \ref{chi}.

\section{} This section presents a proof of Lemma \ref{chi} in  the case when 
 $F$ is a local non-archimedian field. In this case  Lemma \ref{chi}
 is equivalent to the following statement.

\begin{lemma} \begin{enumerate}
\item For any $\chi \neq \| \|^{ -1/2} $
there exists 
$h_\chi \in \mcS (\fg _{ell}^\vee)$ such that 
$ \mcF (h _\chi ) (0)=0$ and   $\kk _\chi (\mcF (h _\chi ))\neq 0$ 

\item  There exists 
$h \in \mcS (\fg _{ell}^\vee)$ such that 
$ \mcF (h) (0)=0$ and  
$\kk _{  \| \|^{ 1/2}}(\mcF (h))\not \in V_0$. \end{enumerate}

\end{lemma}

\begin{proof} Let   $x= \begin{pmatrix}0&1 \\\ep &0 \end{pmatrix},  \ep \in F\sm F^2
\in \fg ^\vee _{ell}$ and $z_u:= \begin{pmatrix}0&0 \\u&0 \end{pmatrix}\in \mcN
,u \in F^\ast $.

\begin{definition} For a  function
$\phi \in \mcS (F) $
we denote
\begin{enumerate}
\item  by $\ti \phi \in \mcS (F)$ the Fourier transform of $\phi$, 
\item 
by $\tau (\phi)$ the distribution of $\mcS (sl_2)$ such that  
 $\tau (\phi)(h)=\int h(ax) \phi (a)da$,
\item by $\ti \tau (\phi) $ the Fourier transform of $ \tau (\phi) $ which is a smooth function on $\fg$.
\item by 
 $\hat \phi$ is the restriction of $\ti \tau (\phi) $ onto  $\mcN$.

 \end{enumerate}
\end{definition}

\begin{claim} \begin{enumerate}

\item $ \hat  \phi (z_u) = \ti \phi (u)$.
\item If $\int \phi (u)du=0$ then 
$\hat \phi \in \mcS (\mcN)$\end{enumerate}
\end{claim}
\begin{proof} The part $(1)$ follows from definition. 

Since $ \hat  \phi (0)=0$ and $tr (xn)\neq 0, n\in \mcN$ the part $(2)$ follows from Appendix B.2.4 of \cite{KV}.
\end{proof}

Since  $\bar \phi \in \mcS (\mcN)$, 
there exists an open compact subgroup $K\subset PGL(2)$ such that $\bar  \phi $ is $K$-invariant. Let $h _\phi := \int _{k\in K} Ad (k) \mcF ( \phi) dk$.

\begin{claim}\label{as2} \begin{enumerate}

\item $h_\phi \in \mcS (\fg _{ell})$.
\item $\mcF (h_\phi )_{|\bO}= \hat  \phi $.
\item $\kk _\chi ( \hat  \phi)(z_1) =\int _a \ti \phi (a)\chi ^{-1} (a)\|a\|^{-1/2}da$. 

\end{enumerate}
\end{claim}

It is clear that for any  $\chi \neq \| \|^{ -1/2} $ we can find $\phi _\chi \in \Phi $ such that  $\kk _\chi (\mcF (\phi _\phi))\neq 0$.  It is easy find  $\phi  \in \Phi $ such 
 that 
$\kk _{  \| \|^{ 1/2}}(\hat \phi)\not \in V_0$.

\end{proof}

\section{}I expect  Conjecture \ref{Lie} to be  a special case of a much more general result on the Fourier transform and will ask about the validity of two such generalizations. Even if 
positive answers to these questions
do not imply  the validity of Conjecture \ref{Lie} one can modify the questions to include the Conjecture. 

 \subsection{}

Let $V$ be a finite-dimensional $\mC$-vector space $P: V \to \mC ,Q :V^\vee \to \mC$ be 
irreducible polynomials of same degrees such that
 $X :=P^{-1}(0)$ does not contain  lines. We choose a non-empty open subset $U$ of $\mC$ and  define a map 
$\kk : \mcS ( Q^{-1}(U))\to \mcS (X) $ as the composition of the Fourier trasform and the restriction onto $X$.

\begin{question} Is the image  of $\kk$  dense in 
$\mcS  (X) $?
\end{question}

\subsection{} Let $F$ be a finite field, $F_n/F$  the degree $n$ extension  and $\bar F$ be the algebraic closure of $F$. Let  $\un V$ be an $F$-vector space and  $P: \un V \to \mA ^r$ be a flat surjective morphism such that subvariety  $\un X := P^{-1}(0)$ of $\un V$ does not contain  images of  non-constant  affine 
$\bar F$-morphisms  $ \mA  \to \un V$. We write  $V_ n:= \un V (F_n)$ and $  X_ n:= \un X (F_n) $

\begin{question} Is it true that for any flat map $ Q :\un V ^\vee \to \mA ^r, deg (Q)=deg(P)$ morphism with 
absolutely irreducible fibers and 
 a non-empty  open  subset $\un U $ of $ \mA ^r$ there exists $n_0$ such that for  $n>n_0$
every function on $X_n$ is the restriction of the Fourier transform of some function supported on $Q^{-1}(\un U)(F_n)$ ?
\end{question}

The research was partially
supported by ERC grant 101142781.

\end{document}